\documentclass[aos,preprint]{imsart}

\RequirePackage[OT1]{fontenc}
\RequirePackage{amsthm,amsmath}
\RequirePackage[numbers]{natbib}
\RequirePackage[colorlinks,citecolor=blue,urlcolor=blue]{hyperref}
\usepackage{amssymb}
\usepackage{bbm}
\usepackage{amsmath}
\usepackage{amsbsy,amsthm}
\usepackage{amscd}

\usepackage{tcolorbox}

\startlocaldefs
\numberwithin{equation}{section}
\theoremstyle{plain}

\newcommand{\R}{\mathbb{R}}
\newcommand{\Rd}{\R^d}
\newcommand{\inr}[1]{\left\langle #1 \right\rangle}

\newcommand{\Z}{\mathbb{Z}}
\newcommand{\E}{\mathbb{E}}

\newcommand{\eps}{\varepsilon}

\newcommand{\cE}{{\cal E}}

\newtheorem{lemma}{Lemma}
\newtheorem{theorem}{Theorem}

\newtheorem{proposition}{Proposition}

\numberwithin{equation}{section}

\def \proof {\noindent {\bf Proof.}\ \ }
\def \remark {\noindent {\bf Remark.}\ \ }

\def \endproof
{{\mbox{}\nolinebreak\hfill\rule{2mm}{2mm}\par\medbreak}}
\def\IND{\mathbbm{1}}

\newcommand{\ol}{\overline}
\newcommand{\wt}{\widetilde}
\newcommand{\wh}{\widehat}

\newcommand{\X}{\mathcal{X}}

\newcommand{\EXP}{\mathbb{E}}
\newcommand{\PROB}{\mathbb{P}}

\newcommand{\Tr}{\mathrm{Tr}}

\newcommand{\defeq}{\stackrel{\mathrm{def.}}{=}}

\endlocaldefs

\begin{document}

\begin{frontmatter}
\title{Robust multivariate mean estimation: the optimality of trimmed mean}
\runtitle{Robust multivariate mean estimation}

\begin{aug}
\author{\fnms{G\'abor} \snm{Lugosi}\thanksref{t1,m1,m2,m3}\ead[label=e1]{gabor.lugosi@upf.edu}},
\and
\author{\fnms{Shahar} \snm{Mendelson}\thanksref{m4}\ead[label=e2]{shahar.mendelson@gmail.com}}

\thankstext{t1}{G\'abor Lugosi was supported by
the Spanish Ministry of Economy and Competitiveness,
Grant PGC2018-101643-B-I00;
``High-dimensional problems in structured probabilistic models - Ayudas Fundaci\'on BBVA a Equipos de Investigaci\'on Cientifica 2017'';
and Google Focused Award ``Algorithms and Learning for AI''.
}
\runauthor{G. Lugosi and S. Mendelson}

\affiliation{ICREA\thanksmark{m1}, Pompeu Fabra University\thanksmark{m2}, Barcelona GSE\thanksmark{m3},
The Australian National University\thanksmark{m4}}

\address{
ICREA \\ 
Pg. Lluís Companys 23 \\ 
08010 Barcelona, Spain \\
and Department of Economics and Business \\ 
Pompeu  Fabra University, Barcelona, Spain \\
\printead{e1}}
\address{
 Mathematical Sciences Institute, \\
The Australian National University \\
Canberra, Australia \\
\printead{e2}}
\end{aug}

\begin{abstract}
We consider the problem of estimating the mean of a random vector based on i.i.d. observations and adversarial contamination. We introduce a multivariate extension of the trimmed-mean estimator and show its optimal performance under minimal conditions.
\end{abstract}



\end{frontmatter}

\section{Introduction}
Estimating the mean of a random vector based on independent and
identically distributed samples is one of the most basic
statistical problems. In the last few years the problem has attracted
a lot of attention and important advances have been made both in terms
of statistical performance and computational methodology.

In the simplest form of the mean estimation problem, one wishes to
estimate the expectation $\mu = \EXP X$ of a random vector
$X$ taking values in $\R^d$, based on a sample $X_1,\ldots,X_N$
consisting of independent copies of $X$. An \emph{estimator} is a
(measurable) function of the data
$$
\wh{\mu} = \wh{\mu}(X_1,\ldots,X_N) \in \R^d~.
$$
We measure the quality of an estimator by the distribution of its
Euclidean distance to the mean vector $\mu$. More precisely,
for a given $\delta >0$---the \emph{confidence parameter}---, one would like to ensure that
$$
\|\wh{\mu} - \mu\| \leq \eps(N,\delta) \ \ {\rm with \ probability \ at \ least \ } 1-\delta
$$
with $\eps(N,\delta)$ as small as possible. Here and in the entire
article, $\|\cdot\|$ denotes the Euclidean norm in $\R^d$.

The obvious choice of $\wh{\mu}$ is the empirical mean
$N^{-1}\sum_{i=1}^N X_i$, which, apart from its
computational simplicity, has good statistical properties when the
distribution is sufficiently well-behaved.
However, it is well known that, even when $X$ is real valued, the
empirical mean behaves sub-optimally and much better mean estimators
are available\footnote{We refer the reader to the recent survey \cite{LuMe19} for an extensive discussion.}. The reason for the
suboptimal performance of the empirical mean is the damaging effect of \emph{outliers} that are inevitably present when the distribution is heavy-tailed.

Informally put, outliers are sample points that are, in some sense,
atypical; as a result they cause a significant distortion to the
empirical mean. The crucial fact is that when $X$ is a heavy-tailed random variable, a typical sample contains a significant number of outliers, implying the empirical mean is likely to be distorted.

To exhibit the devastating effect that outliers cause, let $\eps>0$ and
note that there is a square integrable (univariate) random variable $X$ such that
$$
\left|\frac{1}{N}\sum_{i=1}^N X_i - \mu \right| \geq \eps \ \ \ {\rm
  with \ probability \ at \ least} \ c \frac{\sigma_X^2}{\eps^2 N}
$$
for a positive absolute constant $c$; $\sigma_X^2$ is the
variance of $X$.
In other words, the best possible error $\eps(N,\delta)$   that can be guaranteed by the
empirical mean (when only finite variance is assumed)
is of the order of
$\sigma_X/\sqrt{\delta N}$.
On the other hand, it is well known (see, e.g., the survey
\cite{LuMe19})
that there are estimators of the mean $\wh{\mu}$ such that for all square-integrable random variables $X$,
\begin{equation} \label{eq:intro-opt-single}
|\wh{\mu}-\mu| \leq c\sigma_X \sqrt{\frac{\log(2/\delta)}{N}} \ \ \ {\rm with \ probabilty \ } 1-\delta
\end{equation}
where $c$ is a suitable absolute constant.
An estimator that performs with an error $\eps(N,\delta)$ of the
order of $\sigma_X \sqrt{\log(2/\delta)/N}$ is called a
sub-Gaussian estimator.
Such estimators are optimal in the sense that no estimator can perform with a better error $\eps(N,\delta)$ even if $X$ is known to be a Gaussian random variable.

Because the empirical mean is such a simple estimator and
seeing that outliers are the probable cause of its sub-optimality,
for real-valued random variables, a natural attempt to improve the performance of the empirical mean is
removing possible outliers using a truncation of $X$. Indeed, the
so-called \emph{trimmed-mean} (or \emph{truncated}-mean) estimator
 is defined by removing a fraction of the sample,
consisting of the $\gamma N$ largest and smallest points for some
parameter $\gamma \in (0,1)$, and then averaging over the rest.
This idea is one of the most classical tools in robust statistics and
we refer to  Tukey and McLaughlin \cite{TuMc63}, Huber and Ronchetti
\cite{HuRo09}, Bickel \cite{Bic65}, Stigler \cite{Sti73} for early
work on the theoretical properties of the trimmed-mean estimator.
However, the non-asymptotic sub-Gaussian
property of the trimmed mean was established only recently, by Oliveira and Orenstein in \cite{OlOr19}. They proved that if $\gamma = \kappa \log(1/\delta)/N$
for a constant $\kappa$,  then the trimmed mean estimator $\wh{\mu}$
satisfies (\ref{eq:intro-opt-single}) for all distributions with a
finite variance $\sigma_X$ and with a constant $c$ that depends on
$\kappa$ only.

An added value of the trimmed mean is that it seems to be robust to malicious noise, at least intuitively. Indeed, assume that an
adversary can
corrupt $\eta N$ of the $N$ points for some $\eta <1$. The trimmed-mean estimator can
withstand at least one sort of contamination: the adversary making the
corrupted points either very large or very small. This does not rule
out other damaging changes to the sample, but at least it gives the
trimmed mean another potential edge over other estimators.
And, in fact, as we prove in this article, the performance of the
trimmed-mean estimator is as good as one can hope for
under both heavy-tailed distributions and adversarial corruption.
We show that---a simple variant of---the trimmed-mean estimator
achieves
\begin{equation} \label{eq:intro1}
   |\wh{\mu} - \mu| \le c\sigma_X \left(\sqrt{\eta} +
     \sqrt{\frac{\log(1/\delta)}{N}} \right)
\end{equation}
with probability $1-\delta$, for an absolute constant $c$
(see Theorem \ref{thm:trimmed-mean-dim-1} for the detailed statement).
The bound \eqref{eq:intro1} holds for all univariate distributions with a finite variance, and is minimax optimal in that
class of distributions. For distributions with lighter tail, the
dependence  on the contamination level $\eta$ can be improved.
For example, for sub-Gaussian distributions $\sqrt{\eta}$ may be
replaced by $\eta\sqrt{\log(1/\eta)}$ and the trimmed-mean estimator
achieves that. As we explain in what follows, the parameter $\gamma$ that determines the level of trimming depends on the confidence parameter $\delta$ and contamination level $\eta$ only.

The problem of mean estimation in the multivariate case (i.e., when
$X$ takes values in $\Rd$ for some $d>1$) is considerably more
complex. For i.i.d.\ data without contamination, the best possible statistical performance for square-integrable random vectors is well understood:
if $\Sigma=\EXP \left[(X-\mu)(X-\mu)^T  \right]$ is the covariance
matrix of $X$ whose largest eigenvalue and trace are denoted by
$\lambda_1$ and $\Tr(\Sigma)$, respectively, then for every
$\delta>0$, there exists a mean estimator $\wh{\mu}$ such that,
regardless of the distribution, with probability at least $1-\delta$,
\begin{equation}
\label{eq:subgaussiid}
   \|\wh{\mu} - \mu\| \le c\left( \sqrt{\frac{\Tr(\Sigma)}{N}} +
       \sqrt{\frac{\lambda_1\log(1/\delta)}{N}} \right)
\end{equation}
for some absolute constant $c$. This bound is optimal in the sense
that one cannot improve it even when the distribution is known to be
Gaussian. The existence of such a ``sub-Gaussian'' estimator was
established by Lugosi and Mendelson \cite{LuMe16a}. Computationally
efficient versions have been subsequently constructed by Hopkins
\cite{Hop18} and by Cherapanamjeri, Flammarion, and Bartlett
\cite{ChFlBa19}, see also Depersin and Lecu\'e \cite{DeLe19}.
Once again, we refer to the survey \cite{LuMe19} for related results.

A natural question is how well one can estimate the mean of a random
vector in the presence of adversarial contamination. In particular,
one may ask the following:

\begin{tcolorbox}
Let $X$ be a random vector in $\R^d$ whose mean and covariance matrix
exist. Let $X_1,\ldots,X_N$ be i.i.d.\ copies of $X$.
Then the adversary, maliciously (and knowing in advance of
statistician's intentions), is free to change at most $\eta N$
of the sample points. How accurately can $\mu=\E X$ be estimated with
respect to the Euclidean norm?
In particular, given $\delta$ and $\eta$, does there exist an estimator and an absolute constant
$c$ such that, regardless of the distribution of $X$, with probability
at least $1-\delta$,
\begin{equation}
\label{eq:subgauss}
   \|\wh{\mu} - \mu\| \le c\left( \sqrt{\frac{\Tr(\Sigma)}{N}} +
       \sqrt{\frac{\lambda_1\log(1/\delta)}{N}} + \sqrt{\lambda_1\eta}
     \right)~?
\end{equation}

\end{tcolorbox}

The main result of this article, Theorem \ref{thm:main}, answers this
question in the affirmative. To that end, we construct a
procedure, based on the one-dimensional trimmed-mean estimator,
that has the desired performance guarantees.

\subsection*{Related work}
The model of estimation under adversarial contamination has been
extensively addressed in the literature of computational learning
theory. Its origins may be traced back to the malicious noise model of
Valiant \cite{Val85} and Kearns and Li \cite{KeLi93}. In the context of
mean estimation it has been investigated by
Diakonikolas, Kamath, Kane, Li, Moitra, and Stewart
\cite{DiKaKaLiMoSt16,DiKaKaLiMoSt17,DiKaKaLiMoSt18},
Steinhardt, Charikar, and Valiant \cite{StChVa17}, Minsker \cite{Min18b}.
In particular, in \cite{DiKaKaLiMoSt17} it is shown that when
$N=\Omega((d/\eta)\log d)$ and $\lambda_1$ is the largest eigenvalue
of the covariance matrix $\Sigma$ of $X$, then there exists a
computationally efficient estimator of the mean that satisfies
\[
   \|\wh{\mu}-\mu\| \le c \sqrt{\lambda_1\eta}
\]
with probability at least $9/10$ for all distributions. Although this bound is sub-optimal in terms of the conditions and does not recover the
sub-Gaussian bounds, the goal in \cite{DiKaKaLiMoSt17}, and in other articles in this direction as well, was mainly on computational efficiency. In contrast, our aim is to construct
an estimator with optimal statistical performance, and the multivariate
estimator we propose is not computationally feasible---at least in its
naive implementation---in the sense that computing the estimator takes
time that is exponential in the dimension. It is an intriguing problem to
find computationally efficient mean estimators that have optimal
statistical performance under the weakest possible assumptions: although
such estimators are available for i.i.d.\ data
from the results of Hopkins \cite{Hop18}
and Cherapanamjeri, Flammarion, and Bartlett \cite{ChFlBa19}, these
estimators
are not expected to perform well under adversarial contamination.

The sub-Gaussian estimators achieving the bound  (\ref{eq:subgaussiid})
are based on median-of-means estimators. Such estimators have been
studied under a (somewhat more restrictive) adversarial contamination model
by Lecu{\'e} and Lerasle \cite{LeLe17a} and by Minsker \cite{Min18b},
see also see Rodriguez and Valdora \cite{RoVa19}.
In particular,
Minsker \cite{Min18b} studies estimators that cleverly combine  Huber's robust
$M$-estimators with the  median-of-means technique. His results imply
a performance bound exactly of the form of (\ref{eq:subgauss}). A
disadvantage  of Minsker's estimator is that it assumes that 
the trace and operator norm of the covariance matrix are known up to
a constant factor.

In a recent manuscript, Depersin and Lecu\'e \cite{DeLe19} study the
problem of robust mean estimation a slightly more restrictive model of
contamination. Their main result is a computationally efficient
multivariate mean estimator that achieves a performance similar to
\eqref{eq:subgauss}, though only when $\eta$ is at most a small constant times
$\log(1/\delta)/N$; thus, it is only able to handle low levels of contamination.

Chen, Gao, and Ren \cite{ChGaRe16} develop a general theory of minimax bounds 
under Huber's contamination model (i.e., when the contamination is i.i.d.) for parametric families of distributions. 
In \cite{ChGaRe18} the same authors study robust estimation of the mean vector and covariance matrix
under Huber's contamination model and derive sharp minimax bounds for Gaussian, and more generally
elliptical, distributions. In particular, they show that if the uncontaminated data is Gaussian with
identity covariance matrix, then Tukey's median $\wh{\mu}$ satisfies that, with probability at least $1-\delta$,
\[
   \|\wh{\mu}-\mu\| \le c\left( \sqrt{\frac{d}{N}} +
       \sqrt{\frac{\log(1/\delta)}{N}} + \eta
     \right)~.
\]
Moreover, they prove that this estimator is minimax optimal up to constant factors. 
Note that (\ref{eq:subgauss}) has a similar form except that the term $\eta$ is replaced 
by the weaker $\sqrt{\eta}$. It is remarkable that this is the only (necessary) price
one has to pay for moving from Gaussian  distributions to arbitrary ones whose covariance matrix exists
and from Huber's contamination to adversarial one.
Moreover, as we argue below, for sub-Gaussian distributions the term $\sqrt{\eta}$ may be improved to
$\eta\sqrt{\log(1/\eta)}$. We also refer to  Dalalyan and Thompson \cite{DaTo19} for recent related work.

The rest of the article is organized as follows. In Section
\ref{sec:dim1} we discuss the univariate case and establish
a performance bound for a version of the trimmed-mean estimator
in Theorem \ref{thm:trimmed-mean-dim-1}. We argue that this bound is
best possible up to the value of the absolute constant.
In Section \ref{sec:multivariate} we extend the discussion to the
multivariate case, and  construct a new estimator.
The proof of the performance bound of the multivariate estimator is
given in Section \ref{sec:proof}.

\section{The real-valued case}
\label{sec:dim1}

Let $X$ be a real-valued random variable that has finite variance $\sigma_X^2$.
Set $\mu = \E X$ and define $\ol{X}=X-\mu$.
In what follows, $c,C$
denote positive absolute constants whose value may change at each appearance.
 For $0<p<1$, define the quantile
\begin{equation}
\label{eq:qunatile}
Q_p(\ol{X}) = \sup \left\{M \in \R : \PROB\left(\ol{X} \geq M\right) \geq 1-p\right\}~.
\end{equation}
For simplicity of presentation, we assume throughout the
article that $X$ has an absolutely continuous distribution.
Under this assumption, it follows that $\PROB\left(\ol{X} \geq Q_p(\ol{X})\right) = 1-p$.
However, we
emphasize that this assumption is not restrictive: one may easily
adjust the proof to include all distributions with a finite second moment.
Another solution is that the
statistician can always add a small independent Gaussian noise to the sample points, thus ensuring that the distribution has a density and
without affecting statistical performance.

For reasons of comparison, our starting point is a simple lower bound that limits the performance of
every mean estimator. Similar arguments appear in
\cite{DiKaKaLiMoSt17} and \cite{Min18b}.

While the adversary has total freedom to change at most $\eta N$ of
the sample points,
consider first a rather trivial action: changing the i.i.d.\ sample $(\X_i)_{i=1}^N$ to $(\wt{X}_i)_{i=1}^N$ defined by
\begin{equation} \label{eq:trivial-change}
\wt{X}_i = \min\{X_i, \mu+Q_{1-\eta/2}(\ol{X})\}~.
\end{equation}
Since
$$
\PROB\left(\ol{X} \geq Q_{1-\eta/2}(\ol{X})\right) = \frac{\eta}{2}~,
$$
by a binomial tail bound, with probability at least $1-2\exp(-c\eta N)$,
$$
\left|\left\{i: X_i -\mu \geq Q_{1-\eta/2}(\ol{X})\right\}\right| \leq \frac{3}{4}\eta N~.
$$
In particular, on this event, the adversary can change all sample
points $X_i$ that are bigger than $\mu+Q_{1-\eta/2}(\ol{X})$.
As a result, there is no way one can determine whether $(\wt{X}_i)_{i=1}^N$ is a corrupted sample, originally selected according to $X$ and then changed as in \eqref{eq:trivial-change}, or an uncorrupted sample selected according to the random variable
$$
Z=\min\left\{X, \mu+Q_{1-\eta/2}(\ol{X})\right\}~.
$$
Therefore, on this event, no procedure can distinguish between $\E X$ and $\E Z$, which means that the error caused by this action is at least $| \E Z - \mu|$. Note that for $M=Q_{1-\eta/2}(\ol{X})$ one has that
\[
| \E Z - \mu | =
 \E \left[ (\ol{X}-M) \IND_{\ol{X} \geq M}\right]~.
\]
Since the adversary can target the lower tail of $X$ in exactly the
same way, it follows that,  with probability at least $1-2\exp(-c\eta N)$, no estimator can perform with accuracy better than
\begin{eqnarray*}
\lefteqn{
\ol\cE (\eta,X)   } \\
& \defeq & \max \left\{\E \left[|\ol{X}-Q_{\eta/2}(\ol{X})| \IND_{\ol{X} \leq Q_{\eta/2}(\ol{X})}\right], \E \left[|\ol{X}-Q_{1-\eta/2}(\ol{X})| \IND_{\ol{X} \geq Q_{1-\eta/2}(\ol{X})} \right] \right\}~.
\end{eqnarray*}
Of course, the adversary has a second trivial action: do nothing. That
is a better corruption strategy (in the minimax sense) when
$$
\ol\cE(\eta,X)  \leq C \sigma_X \sqrt{\frac{\log(2/\delta)}{N}}~.
$$
Therefore, if one wishes to find a procedure that performs with probability at least $1-\delta-2\exp(-c\eta N)$, the best error one can hope for is
\begin{equation} \label{eq:corrupt-error}
\ol\cE(\eta,X) + C \sigma_X\sqrt{\frac{\log(2/\delta)}{N}}~,
\end{equation}
where $c$ and $C$ are absolute constants.

A rather surprising fact is that in the real-valued case, the two trivial actions cause the largest possible damage. Indeed,
we show that there is an estimator that is a simple modification of
trimmed mean that attains what is almost the optimal
error---with $\ol\cE(\eta,X)$ replaced by
$$
\cE (\eta,X)\defeq \max \left\{\E \left[|\ol{X}| \IND_{\ol{X} \leq Q_{\eta/2}(\ol{X})}\right], \E \left[|\ol{X}| \IND_{\ol{X} \geq Q_{1-\eta/2}(\ol{X})} \right] \right\}~.
$$

\remark
It is straightforward to construct a random variable $X$ for which
$\ol\cE(\eta,X) \geq c_1 \sqrt{\eta}\sigma_X$.
(Take, for example $X$ that takes value $0$ with probability $1-\eta$ and values $\pm\sigma_X/\sqrt{\eta}$
with probability $\eta/2$ each.)
Thus, in terms of $\eta, \sigma_X, \delta$ and $N$, the best minimax error rate that is possible in the corrupted mean estimation problem for real-valued random variables is
$$
c \sigma_X \max\left\{\sqrt{\eta}, \sqrt{\frac{\log(2/\delta)}{N}}\right\}
$$
for a suitable absolute constant $c$.

Next, let us define the modified trimmed-estimator. The estimator splits the data into two
equal parts. Half of the data points are used to determine the truncation
at the appropriate level. The points from the other half are
averaged as is, except for the data points
that fall outside of the estimated quantiles, which are truncated prior to averaging.
For convenience, assume that the
data consists of $2N$ independent copies of the random variable $X$,
denoted by  $X_1,\ldots,X_N,Y_1,\ldots,Y_N$. The statistician has access
to the corrupted sample
$\wt{X}_1,\ldots,\wt{X}_N,\wt{Y}_1,\ldots,\wt{Y}_N$,
where at most $2\eta N$ of the sample points have been changed by an adversary.

 For $\alpha \leq \beta$, let
\[
\phi_{\alpha,\beta}(x) =
\begin{cases}
\beta & \mbox{if} \  x > \beta,
\\
x & \mbox{if} \ x \in [\alpha,\beta],
\\
\alpha & \mbox{if} \ x < \alpha,
\end{cases}
\]
and for $x_1,\ldots, x_m \in \R$ let $x_1^* \leq x_2^* \leq \cdots \leq  x_m^*$ be its non-decreasing rearrangement.

With this notation in place, the definition of the estimator is as follows:

\begin{tcolorbox}
{\bf Univariate mean estimator.}
\begin{description}
\item{$(1)$} Consider the corrupted sample
  $\wt{X}_1,\ldots,\wt{X}_N,\wt{Y}_1,\ldots,\wt{Y}_N$ as input.
\item{$(2)$} Given the corruption parameter $\eta$ and confidence level $\delta$, set
$$
\eps=8\eta + 12\frac{\log(4/\delta)}{N}~.
$$
\item{$(3)$} Let $\alpha=\wt{Y}_{\eps N}^*$ and $\beta=\wt{Y}_{(1-\eps) N}^*$ and set
$$
\wh{\mu} =\frac{1}{N}\sum_{i=1}^N \phi_{\alpha,\beta}(\wt{X}_i)~.
$$
\end{description}
\end{tcolorbox}

\begin{theorem} \label{thm:trimmed-mean-dim-1}
Let $\delta \in (0,1)$ be such that $\delta \ge e^{-N}/4$.
Then, with probability at least $1-\delta$,
$$
|\wh{\mu}-\mu | \leq 3 \cE(4\eps,X)+ 2\sigma_X
  \sqrt{\frac{\log(4/\delta)}{N}}~.
$$
Moreover, with probability at least $1-4\exp(- \eps N/12)$,
$$
|\wh{\mu}-\mu | \leq 10 \sqrt{\eps} \sigma_X~.
$$
\end{theorem}

\remark
The necessity of prior knowledge of the confidence parameter
$\delta$ was pointed out (even in the contamination-free case) by
Devroye, Lerasle, Lugosi, and Oliveira \cite{DeLeLuOl16}, see
\cite{LuMe19} for further discussion. The contamination level
need not be known exactly. If an upper bound $\ol{\eta}\ge \eta$ is
available and one uses the estimator with parameter $\ol\eta$ instead
of $\eta$, then the same bound holds with $\eta$ replaced by $\ol{\eta}$.

To explain the meaning of Theorem \ref{thm:trimmed-mean-dim-1}, observe that for $M=Q_{1-\eps/2}(\ol{X})$, one has
$$
\frac{\eps}{2} = \PROB\left(\ol{X} \geq M\right) \leq \frac{\sigma_X^2}{M^2}~,
$$
and in particular,
\begin{equation}
\label{eq:quantilebound}
Q_{1-\eps/2}(\ol{X}) \leq  \frac{\sigma_X\sqrt{2}}{\sqrt{\eps}}~.
\end{equation}
Also,
\begin{eqnarray}
\label{eq:ebound}
\E \left[(\ol{X}-M) \IND_{\ol{X} \geq M}\right] & \leq  &
\E\left[|\ol{X}| \IND_{\ol{X} \geq M} \right]+ \E \left[M \IND_{\ol{X} \geq M} \right]
\nonumber \\
& \leq  &
\sigma_X \PROB^{1/2}(\ol{X} \geq M) + |M| \PROB(\ol{X} \geq M)
\nonumber \\
& \leq &
\sigma_X\sqrt{8\eps}~,
\end{eqnarray}
implying that for every $X$,
\begin{equation} \label{eq:upper-error-intro}
\cE(\eps,X) \leq \sigma_X\sqrt{8\eps}~.
\end{equation}
 Hence, Theorem \ref{thm:trimmed-mean-dim-1} shows that the
 estimator attains the minimax rate
of the corrupted mean-estimation problem, noted previously.

Of course, Theorem \ref{thm:trimmed-mean-dim-1} actually implies sharper individual bounds: if $\eta N \leq \log(2/\delta)$, then $\eps \sim N^{-1} \log(2/\delta)$ and the assertion of Theorem \ref{thm:trimmed-mean-dim-1} is that, with probability at least $1-\delta$,
$$
|\wh{\mu}-\mu | \le C \sigma_X \sqrt{\frac{\log(2/\delta)}{N}}~,
$$
which matches the optimal sub-Gaussian error rate. If, on the other hand, $\eta N > \log(2/\delta)$, then with probability at least $1-\delta$,
$$
|\wh{\mu}-\mu | \le C \cE(c\eta, X)~,
$$
essentially matching the lower bound (\ref{eq:corrupt-error}).

\remark
Observe that the upper bound
on
 $\cE(\eps,X)$ in
\eqref{eq:upper-error-intro} is based only on $\sigma_X$, and
therefore on the fact that $X$ is square-integrable. Under stronger
moment assumptions on $X$, an improved bound can be easily
established. For example, if $X$ is sub-Gaussian, that is, if for
every  $p \geq 2$, $\left(\EXP |\ol{X}|^p\right)^{1/p} \leq c \sqrt{p}
\sigma_X$,
the same argument used in \eqref{eq:upper-error-intro} for $p=\log(1/\eps)$ shows that
$$
2\cE(4\eps,X)  + \frac{\eps}{2} \max\{|Q_{\eps/2}(\ol{X})|, |Q_{1-\eps/2}(\ol{X})|\} \leq c \eps \sqrt{\log(1/\eps)} \sigma_X~.
$$

One may wonder if $\eta \sqrt{\log(1/\eta)}$ is the correct order of dependence 
on the contamination level for sub-Gaussian distributions.  
As it is proved by Chen, Gao, and Ren \cite{ChGaRe18},
if $X$ is Gaussian and the contamination comes from Huber's model, 
the correct dependence on the contamination level is proportional to $\eta$, suggesting 
a possible slight improvement. 
At the same time, as we discuss it above, $\ol{\mathcal{E}}(\eta,X)$ is a lower bound for any estimator.
One may easily check that, if $X$ is Gaussian, $\ol{\mathcal{E}}(\eta,X)$ is of the order of
 $\eta /\sqrt{\log(1/\eta)}$ so this lower bound is loose in this case. 
Interestingly, however, there exist sub-Gaussian distributions under which
$\ol{\mathcal{E}}(\eta,X)$ is of the order of $\eta \sqrt{\log(1/\eta)}$. 
(As an example, one may take $X=\IND_{|G| \le Q} \min(1,|G|) + \IND_{|G| > Q} |G|$ where $G$ is a standard 
Gaussian random variable and $Q$ is its $1-\eta/2$ quantile.)
This means that for
sub-Gaussian distributions, the upper bound of Theorem 1 is indeed tight, up to constant 
factors. 
Note that our lower bound uses the adversarial nature of the contamination, so
it might be the case that under Huber's model, even for sub-Gaussian distributions, 
$\eta$ is the correct order.

\subsection{Proof of Theorem \ref{thm:trimmed-mean-dim-1}}

Recall that one is given the corrupted sample $\wt{X}_1,\ldots,\wt{X}_N,\wt{Y}_1,\ldots,\wt{Y}_N$, out of which at most $2\eta N$ of the sample points have been corrupted. Also, $(z_i^*)_{i=1}^N$ denotes a \emph{non-decreasing} rearrangement of the sequence $(z_i)_{i=1}^N$.

The first step of the estimation procedure determines the
truncation level, which is done using the first half of the
corrupted sample.


Consider the corruption-free sample $Y_1,\ldots,Y_N$ and let
$U=\IND_{\ol{X} \geq Q_{1-2\eps}(\ol{X}) }$. Since $X$ is
absolutely continuous,
we have that $\PROB\left(\ol{X} \geq Q_{1-2\eps}(\ol{X})\right)=2\eps$ and
$$
\sigma_U \leq \PROB^{1/2}(\ol{X} \geq Q_{1-2\eps}(\ol{X})) = (2\eps)^{1/2}~.
$$
A straightforward application of Bernstein's inequality shows that, with probability at least $1-\exp(-\eps N/12)$,
\begin{equation} \label{eq:E1}
\left|\{i : Y_i \geq \mu + Q_{1-2\eps}(\ol{X})\}\right| \geq \frac{3}{2}\eps N~.
\end{equation}
A similar argument for $U=\IND_{\ol{X}  > Q_{1-\eps/2}(\ol{X}) }$ implies that, with probability at least $1-\exp(-\eps N/12)$,
\begin{equation} \label{eq:E2}
\left|\{i : Y_i \leq \mu + Q_{1-\eps/2}(\ol{X})\}\right| \geq (1-(3/4)\eps) N~.
\end{equation}
Similarly, with probability at least $1-2\exp(-\eps N/12)$,
\begin{equation} \label{eq:E3}
\left|\{i : Y_i \leq \mu + Q_{2\eps}(\ol{X})\}\right| \geq \frac{3}{2}\eps N~,
\end{equation}
and, with probability at least $1-2\exp(-\eps N/12)$,
\begin{equation} \label{eq:E4}
\left|\{i : Y_i \geq \mu + Q_{\eps/2}(\ol{X})\}\right| \geq (1-(3/4)\eps) N~.
\end{equation}
Thus, with probability at least $1-4\exp(-\eps N/12)\ge 1-\delta/2$, \eqref{eq:E1}--\eqref{eq:E4} hold simultaneously on an event we denote by $E$. Importantly, the event $E$ only depends on the uncorrupted sample $Y_1,\ldots,Y_N$.

Since $\eta \le \epsilon/8$, following any corruption of at most $2\eta N$ points, on the event $E$
$$
\left|\{i : \wt{Y}_i \geq \mu + Q_{1-2\eps}(\ol{X})\}\right| \geq \left((3/2)\eps-2\eta\right) N \geq \eps N
$$
and
$$
\left|\{i : \wh{Y}_i \leq \mu + Q_{1-\eps/2}(\ol{X})\}\right| \geq (1-(3/4)\eps-2\eta) N \geq (1-\eps) N~;
$$
in other words,
\begin{equation}\label{eq:beta-single}
Q_{1-2\eps}(\ol{X}) \leq \wt{Y}_{(1-\eps) N}^* - \mu \leq  Q_{1-\eps/2}(\ol{X})~.
\end{equation}
Similarly, on the event $E$, we also have
\begin{equation} \label{eq:alpha-single}
Q_{\eps/2}(\ol{X}) \leq \wt{Y}_{\eps N}^*-\mu \leq Q_{2\eps}(\ol{X})~.
\end{equation}
Recall that the truncation levels are
\[
\alpha=\wt{Y}_{\eps N}^* \ \ \ {\rm and } \ \ \ \beta=\wt{Y}_{(1-\eps) N}^*~.
\]
To prove Theorem \ref{thm:trimmed-mean-dim-1},
first we show that $(1/N)\sum_{i=1}^N \phi_{\alpha,\beta}(X_i)$
satisfies an inequality of the wanted form, and then we prove that
corruption does not change the empirical mean of
$\phi_{\alpha,\beta}$ by too much; that is, that
$$
\left| \frac{1}{N} \sum_{i=1}^N \phi_{\alpha,\beta}(X_i) - \frac{1}{N} \sum_{i=1}^N \phi_{\alpha,\beta}(\wt{X}_i) \right|
$$
is also small enough.

For the first step, note that on the event $E$,
\begin{eqnarray}
\label{eq:sum-decomp}
   \frac{1}{N} \sum_{i=1}^N \phi_{\alpha,\beta}(X_i)   
& \le &
\frac{1}{N} \sum_{i=1}^N \phi_{\mu+Q_{2\epsilon}(\ol{X}),\mu+Q_{1-\epsilon/2}(\ol{X})}(X_i)    \\
\nonumber
& = &
\EXP \phi_{\mu+Q_{2\epsilon}(\ol{X}),\mu+Q_{1-\epsilon/2}(\ol{X})}(X) \\
& &
\nonumber 
+
\frac{1}{N} \sum_{i=1}^N \left( \phi_{\mu+Q_{2\epsilon}(\ol{X}),\mu+Q_{1-\epsilon/2}(\ol{X})}(X_i) - \right. \\
& &
\nonumber 
 \qquad \left. 
\EXP \phi_{\mu+Q_{2\epsilon}(\ol{X}),\mu+Q_{1-\epsilon/2}(\ol{X})}(X) \right)~.
\end{eqnarray}
The first term on the right-hand side of \eqref{eq:sum-decomp} is bounded by
\begin{eqnarray*}
\EXP \phi_{\mu+Q_{2\epsilon}(\ol{X}),\mu+Q_{1-\epsilon/2}(\ol{X})}(X)
& \leq &  \mu + 
\E \left[\ol{X} \IND_{\ol{X} \geq
         Q_{1-\eps/2}(\ol{X})} \right] \\
& \le &  \mu +\cE(\eps,X)~.
\end{eqnarray*}
On the other hand, since
\begin{eqnarray*}
\EXP \phi_{\mu+Q_{2\epsilon}(\ol{X}),\mu+Q_{1-\epsilon/2}(\ol{X})}(X)
& \geq &  \mu - 
\E \left[\ol{X} \IND_{\ol{X} \leq
         Q_{2\eps}(\ol{X})} \right] \\
& \ge &  \mu -\cE(4\eps,X)~,
\end{eqnarray*}
the second term on the right-hand side of \eqref{eq:sum-decomp}
 is a sum of centered i.i.d.\ random variables (independent of $E$) that are upper bounded
by $Q_{1-\eps/2}(\ol{X})+\cE(4\eps,X)$ and whose
variance
is at most $\sigma_X^2$. Therefore, by Bernstein's inequality, conditioned on $Y_1,\ldots,Y_n$,
 with probability at least $1-\delta/4$,
\begin{eqnarray*}
\lefteqn{
\frac{1}{N} \sum_{i=1}^N \phi_{\alpha,\beta}(X_i) }  \\
& \leq &
\mu +\cE(\eps,X)  +
\sigma_X \sqrt{\frac{2\log(4/\delta)}{N}}
                                                           +\frac{Q_{1-\eps/2}(\ol{X})\log(4/\delta)}{N}
+\frac{\cE(4\eps,X) \log(4/\delta)}{N}  \\
& \le & \mu +2\cE(4\eps,X)  +
2\sigma_X \sqrt{\frac{\log(4/\delta)}{N}}~,
\end{eqnarray*}
where we used the fact that by \eqref{eq:quantilebound},
$Q_{1-\eps/2}(\ol{X}) \log(4/\delta)/N\le \sigma_X
\sqrt{\frac{\log(4/\delta)}{6N}}$ and that $\cE(4\eps,X)
\log(4/\delta)/N \le \cE(4\eps,X)$ by the assumption that $\delta
\ge e^{-N}/4$.

An identical argument for the lower tail shows that, on the event
$E$, with probability at least $1-\delta/2$,
\begin{eqnarray*}
\left| \frac{1}{N} \sum_{i=1}^N \phi_{\alpha,\beta}(X_i) - \mu \right|  & \leq &
2\cE(4\eps,X)  + 2\sigma_X \sqrt{\frac{\log(4/\delta)}{N}}~.
\end{eqnarray*}
It remains to show that, on the event $E$,
\[
\left| \frac{1}{N} \sum_{i=1}^N \phi_{\alpha,\beta}(X_i) - \frac{1}{N} \sum_{i=1}^N \phi_{\alpha,\beta}(\wt{X}_i) \right|
\]
is small.
Since $\phi_{\alpha,\beta}(X_i) \not = \phi_{\alpha,\beta}(\wt{X}_i)$ for at most $2\eta N$ indices, and for such points that maximal gap is
$$
|\phi_{\alpha,\beta}(X_i) - \phi_{\alpha,\beta}(\wt{X}_i)| \leq
|Q_{\eps/2}(\ol{X})| +|Q_{1-\eps/2}(\ol{X})|~,
$$
it follows that
\begin{eqnarray*}
\left| \frac{1}{N} \sum_{i=1}^N \phi_{\alpha,\beta}(X_i) - \frac{1}{N}
  \sum_{i=1}^N \phi_{\alpha,\beta}(\wt{X}_i) \right|
& \leq &
2\eta \left( |Q_{\eps/2}(\ol{X})|+ |Q_{1-\eps/2}(\ol{X})|\right)  \\
& \le &
 \frac{\eps}{2} \max\{|Q_{\eps/2}(\ol{X})|, |Q_{1-\eps/2}(\ol{X})|\}~,
\end{eqnarray*}
since $\eta \leq \eps/8$.
Finally, note that
\[
\frac{\eps}{2} Q_{1-\eps/2}(\ol{X})= \EXP \left[Q_{1-\eps/2}(\ol{X}) \IND_{\ol{X} \geq
         Q_{1-\eps/2}(\ol{X})} \right]\le
\EXP \left[\ol{X} \IND_{\ol{X} \geq
          Q_{1-\eps/2}(\ol{X})} \right]~,
 \]
and therefore, on the event $E$, we have
\[
  \left| \frac{1}{N} \sum_{i=1}^N \phi_{\alpha,\beta}(X_i) - \frac{1}{N}
  \sum_{i=1}^N \phi_{\alpha,\beta}(\wt{X}_i) \right| \le \cE(\eps,X)~.
\]
The second statement of the theorem now follows by \eqref{eq:upper-error-intro}.
\endproof

\section{Robust multivariate mean estimation}
\label{sec:multivariate}

In this section we present the main findings of the article: we
construct a multivariate version of the robust mean estimator
and establish the corresponding performance bound announced
in the introduction.

As one may expect, the procedure in the multi-dimensional case is
significantly more involved than in dimension one.
In what follows, $X$ is a random vector taking values in $\R^d$ with mean $\mu = \E
X$ and  covariance matrix of $\Sigma$.  As before, we write $\ol{X}=X-\mu$,
$\lambda_1$ denotes the
largest eigenvalue of $\Sigma$, and $\Tr(\Sigma)= \EXP
\left\|\ol{X}\right\|^2$ is its trace.

Recall that a
mean estimator receives as data a sample $(\wt{X}_i)_{i=1}^N$ that an
adversary fabricates by corrupting at most $\eta N$ points of a
sample $X_1,\ldots,X_N$ of independent, identically distributed
copies of the random vector $X$. As in the univariate case, the estimator requires knowledge of
the contamination level $\eta$ and the confidence parameter $\delta$.
Once again, for clarity of the presentation, we assume that $X$ has an
absolutely continuous distribution with respect to the Lebesgue
measure.

\begin{theorem} \label{thm:main}
Assume that $X$ is a random vector in $\Rd$ that has a mean and
covariance matrix.
There exists a mean estimator $\wh\mu$
that takes the parameters $\delta\in (0,1),\eta \in [0,1)$
and the contaminated
data $(\wt{X}_i)_{i=1}^N$ as input, and satisfies that,
with probability at least $1-\delta$,
$$
\|\wh{\mu}-\mu\| \leq c \left(\sqrt{\frac{\Tr(\Sigma)}{N}}
  +\sqrt{\frac{\lambda_1\log(1/\delta)}{N}} + \sqrt{\lambda_1\eta}\right)~,
$$
where $c>0$ is a numerical constant.
\end{theorem}

A value of the numerical constant is explicitly given in the
proof. However, no attempt has been made to optimize its value.

The same remark as in the univariate case  on the previous knowledge
of $\eta$ and $\delta$, mentioned after Theorem
\ref{thm:trimmed-mean-dim-1}, applies here as well.

As it is pointed out in the introduction, the bound of Theorem
\ref{thm:main} coincides with the best possible bound in the
corruption-free case up to the term $\sqrt{\lambda_1\eta}$ that is
the price one has to pay for adversarial corruption. 
The fact that the term $\sqrt{\lambda_1\eta}$ is inevitable in the
upper bound follows from the fact that for
any upper bound for the norm of difference $\|\wh\mu -\mu\|$, the same
upper bound holds for
any one-dimensional marginal. Hence, the necessity of this term follows from our
arguments in the univariate case.
At the same time, similarly to the univariate case,
under higher moment assumptions, the term $\sqrt{\lambda_1 \eta}$
may be improved. For instance, if the distribution is sub-Gaussian (in the sense that all one-dimensional 
projections are sub-Gaussian), then
this term may be replaced by $\eta \sqrt{\log(1/\eta)} \sqrt{\lambda_1} $. 
This may be seen by a straightforward modification of the proof.

Remarkably,
the malicious sample corruption affects only the ``weak'' term of the
bound, that is, it scales with the square root of the operator norm of
the covariance matrix. Indeed, if the corruption parameter $\eta$ is
such that $\eta N \leq \log(2/\delta)$, then, with probability at least $1-\delta$, $\wh{\mu}$ satisfies
\begin{equation} \label{eq:low-corruption-intro}
\|\wh{\mu}-\mu\| \leq c \left(\sqrt{\frac{\Tr(\Sigma)}{N}} + \sqrt{\lambda_1} \sqrt{\frac{\log(1/\delta)}{N}}\right)~,
\end{equation}
matching the optimal bound for multivariate mean estimation bound from
\cite{LuMe16a} for the corruption-free case.
If, on the other hand, the corruption parameter is larger, then Theorem \ref{thm:main} implies that with probability at least $1-2\exp(-\eta N/c)$,
\begin{equation} \label{eq:high-corruption-intro}
\|\wh{\mu}-\mu\| \leq c \left(\sqrt{\frac{\Tr(\Sigma)}{N}} + \sqrt{\eta} \sqrt{\lambda_1} \right)
\end{equation}
for a numerical constant $c>0$.

In what follows we describe the construction of the mean
estimator $\wh\mu$ that satisfies the announced performance bound.

\subsection{The multivariate mean estimator}

The main component is a mean estimation procedure that, in order to
perform well, requires information on $\Tr(\Sigma)$ and
$\lambda_1$. Since such information is not assumed to be available,
we produce an estimator depending on a tuning parameter $Q$.
Then we
use a simple mechanism of choosing
the appropriate value of $Q$.

Just like in the univariate case, for simplicity of
notation, assume that the estimator receives $2N$ data points
$\wt{X}_1,\ldots,\wt{X}_N,\wt{Y}_1,\ldots,\wt{Y}_N$, and that at most $2\eta N$ points of the original independent
sample
$X_1,\ldots,X_N,Y_1,\ldots,Y_N$ have been changed by the adversary.
The procedure computes, for each unit vector $v$ and tuning parameter
$Q>0$,
the trimmed mean
estimate of the expectation of the projection of $X$ to the line
spanned by $v$ with a minor difference: the truncation level is
widened depending on the parameter $Q$. Each one of these estimators defines a slab in $\Rd$. The details are as follows:

\begin{tcolorbox}
{\bf Multivariate mean estimator.}
 \begin{description}
\item{$(1)$}
Set
$$
\eps=\max\left( 10\eta,  2560\frac{\log(2/\delta)}{N}\right)~.
$$

\item{$(2)$} Let $S^{d-1}$ be the Euclidean unit sphere in $\R^d$ and for every $v \in S^{d-1}$ define
$$
\alpha_v = \left(\inr{\wt{Y}_i,v}\right)_{(\eps/2)N}^* \quad {\rm and} \quad \beta_v=\left(\inr{\wt{Y}_i,v}\right)_{(1-\eps/2)N}^*~.
$$
\item{$(3)$} For every $v \in S^{d-1}$ and $Q>0$, set
$$
U_Q(v)=\frac{1}{N} \sum_{i=1}^N \phi_{\alpha_v-Q,\beta_v+Q}\left(\inr{\wt{X}_i,v}\right)~,
$$
and let
$$
\Gamma(v,Q) = \left\{ x \in \R^d : |\inr{x,v}-U_Q(v)| \leq 2\eps
  Q\right \}~.
$$

\item{$(4)$} For each $Q>0$, set
$$
\Gamma(Q)= \bigcap_{v \in S^{d-1}} \Gamma(v,Q)~.
$$

\item{$(5)$} Let $i^* \in \Z$ be the smallest such that
$\bigcap_{i \geq i^*} \Gamma(2^i) \neq \emptyset$. Define $\wh{\mu}$ to be any point in
\[
     \bigcap_{i\in \Z: i \geq i^*}  \Gamma(2^i)~.
\]
\end{description}
\end{tcolorbox}

Each set $\Gamma(Q)$ is an intersection of random slabs, one for each
direction in the sphere $S^{d-1}$. The ``center'' of the slab
associated with the direction $v$ is $U_Q(v)$ and its width is
proportional to $\eps Q$. As we show in what follows, there is some $i_0 \in \Z$ such that with probability at least $1-\delta$, the sets $\Gamma(2^i), \ i \geq i_0$ are nested, implying that $\wh{\mu}$ is well-defined.
Note that the last step of selecting the value of $Q$ is reminiscent of 
Lepski's method \cite{Lep91} or the related method ``intersection of confidence intervals''  by Goldenshluger and Nemirovski \cite{GoNe97}.


\section{Proof of Theorem \ref{thm:main}}
\label{sec:proof}

The heart of the proof of Theorem \ref{thm:main} is the following
proposition that describes the performance
of an estimator with the correct tuning parameter $Q$.

The role of $Q$ is to incorporate the ``global complexity" of
$S^{d-1}$. In particular, if $Q$ is selected properly, that is enough
to ensure that $\Gamma(Q)$ is nonempty and contains a good estimator
of $\mu$. This is formalized in the next proposition.

\begin{proposition} \label{thm:main-vector}
Let
\begin{equation}  \label{eq:choice-of-Q}
Q_0=
\max\left(\frac{256}{\eps}\sqrt{\frac{\Tr(\Sigma)}{N}}, 16\sqrt{\frac{\lambda_1}{\eps}} \right)
\end{equation}
and consider $Q \in [2Q_0,4Q_0]$. Then, with probability at least $1-2\exp(-\eps N/2560) \geq 1-\delta$, $\Gamma(Q) \not = \emptyset$ and for every $z \in \Gamma(Q)$,
$$
\|z-\mu\| \leq 4 \eps Q_0~.
$$
\end{proposition}


Observe that for every $Q$, the diameter of $\Gamma(Q)$ is at most
$4\eps Q$.
Indeed, if $x_1,x_2 \in \Gamma(Q)$ then for every $v \in S^{d-1}$,
$$
|\inr{x_1-x_2,v}|=|\inr{x_1,v}-\inr{x_2,v}| \leq |\inr{x_1,v}-U_Q(v)|+|\inr{x_2,v}-U_Q(v)| \leq 4\eps Q~,
$$
implying that $\|x_1-x_2\| \leq 4 \eps Q$.


The key component in the proof of Proposition \ref{thm:main-vector} is the next lemma.
\begin{lemma} \label{lemma:uniform-cutoff}
For each $i\in \{1,\ldots,N\}$ and $v\in S^{N-1}$, define $\ol{Y}_i(v)=\inr{Y_i-\mu,v}$.
With probability at least $1-\exp(-\eps N/2560) \geq 1-\delta/2$,
\begin{equation} \label{eq:uniform-beta}
\sup_{v \in S^{d-1}} \left|\left\{i : \ol{Y}_i(v) \geq Q_0\right\}\right| \leq
\frac{\eps}{8}N
\quad \text{and} \quad
\sup_{v \in S^{d-1}} \left|\left\{i : \ol{Y}_i(v) \leq - Q_0\right\}\right| \leq \frac{\eps}{8}N~.
\end{equation}
\end{lemma}

Lemma \ref{lemma:uniform-cutoff} is a uniform version of the analogous claim used in the univariate case.

\medskip

\proof
Let us prove the first inequality; the second is proved by an identical argument and is omitted.
Consider the function $\chi: \R\to \R$, defined by
\[
   \chi(x) = \left\{ \begin{array}{ll}
                            0 &  \text{if} \ x\le  Q_0/2, \\
                           \frac{2x}{Q_0} -1 & \text{if} \ x\in
                                               (Q_0/2,Q_0], \\
                            1 &   \text{if} \ x >  Q_0.  \\
                          \end{array} \right.
\]
Observe that $\IND_{\{\ol{Y}(v) \geq Q_0\}}\le \chi(\ol{Y}(v))\le \IND_{\{\ol{Y}(v) \geq Q_0/2\}}$, and that $\chi$ is
Lipschitz with constant $2/Q_0$.
Therefore, if $\eps_1,\ldots,\eps_N$ are independent, symmetric $\{-1,1\}$-valued random variables that are independent of the $(Y_i)_{i=1}^N$, then
\begin{eqnarray*}
\lefteqn{
\E \sup_{v \in S^{d-1}} \frac{1}{N}\sum_{i=1}^N \IND_{\{\ol{Y}_i(v)
  \geq Q_0\}}      } \\
& \leq &
\E \sup_{v \in S^{d-1}} \frac{1}{N}\sum_{i=1}^N \chi(\ol{Y}_i(v)) \\
& \leq &
2\E \sup_{v \in S^{d-1}} \frac{1}{N}\left| \sum_{i=1}^N \eps_i
         \chi(\ol{Y}_i(v)) \right|
+          \sup_{v \in S^{d-1}} \E \chi(\ol{Y}(v))  \\
& & \text{(by the Gin{\'e}-Zinn symmetrization theorem \cite{GiZi84})}
  \\
& \leq &
\frac{4}{Q_0} \E \sup_{v \in S^{d-1}} \frac{1}{N}\left| \sum_{i=1}^N \eps_i
         \ol{Y}_i(v) \right|
+          \sup_{v \in S^{d-1}} \E \chi(\ol{Y}(v))
\\
& \defeq & (*)~,
\end{eqnarray*}
where in the second step one uses the standard contraction lemma for
Rademacher averages, see Ledoux and Talagrand \cite{LeTa91}.

To bound the second term on the right-hand side, recall that $Q_0 \geq 16 \sqrt{\lambda_1/\eps}$, and thus, for every $v \in S^{d-1}$,
\begin{eqnarray}
\label{eq:chebyshev}
\E \chi(\ol{Y}(v)) & \le & \E \IND_{\{\ol{Y}(v) \geq Q_0/2\}}= \PROB\left(
  \inr{\ol{X},v} \geq \frac{Q_0}{2}\right) \\
\nonumber
&  \leq & \frac{4 \E\inr{\ol{X},v}^2}{Q_0^2} \leq \frac{4 \lambda_1}{Q_0^2} \leq \frac{\eps}{64}~.
\end{eqnarray}
To bound the first term, note that
\begin{equation*}
\E \sup_{v \in S^{d-1}} \left| \frac{1}{N}\sum_{i=1}^N \eps_i \ol{Y}_i(v) \right| =  \E \sup_{v \in S^{d-1}} \left| \frac{1}{N}\sum_{i=1}^N \eps_i \inr{X_i-\mu,v} \right|
\leq \sqrt{\frac{\Tr(\Sigma)}{N}}~.
\end{equation*}
Hence, by the definition of $Q_0$,
$$
(*) \leq \frac{\eps}{32}~.
$$
By Talagrand's concentration inequality for empirical processes indexed by a class of uniformly bounded functions \cite{Tal96c}, with probability at least $1-\exp(-x)$,
$$
\frac{1}{N}\sup_{v \in S^{d-1}} \left|\left\{i : \ol{Y}_i(v) \geq
    Q_0\right\}\right| \leq
\frac{\eps}{16}
+ \sqrt{\frac{x}{N}} \cdot \frac{\sqrt{\eps}}{128} + \frac{10x}{N}
$$
(see  \cite[Exercise 12.15]{BoLuMa13} for the value of the numerical constant).

With the choice of $x = \eps N/2560$ one has that, with probability at least $1-\exp(-\eps N/2560)$,
$$
\sup_{v \in S^{d-1}} \left|\left\{i : \ol{Y}_i(v)\geq Q_0\right\}\right| \leq \frac{\eps}{8} N~,
$$
as required.
\endproof

Note that, when \eqref{eq:uniform-beta} holds, we have,
for every $v \in S^{d-1}$,
$$
\alpha_v-\inr{\mu,v} \geq  - Q_0  \ \ \ {\rm and} \ \ \ \beta_v - \inr{\mu,v} \leq Q_0~.
$$
Indeed, this follows from the fact that
for every $v \in S^{d-1}$ there are at most $(\eps/8)N$ of the $\ol{Y}_i(v)$ that
are larger than $Q_0$. If, in addition, the adversary corrupts at most
$(\eps/8) N$ of the points $Y_i$,
then there are still no more than $(\eps/4) N$ values
$\inr{\wt{Y}_i,v}$
that are larger than $\inr{\mu,v}+Q_0$, which suffices for our purposes.
And, by the definition of $\eps$, one has that $\eps/8 \ge \eta$, as required.

Now consider some $Q$ that satisfies $2Q_0<Q \leq 4Q_0$, and
from here we condition on an event $E$ such that the inequalities  \eqref{eq:uniform-beta} both hold. By Lemma \ref{lemma:uniform-cutoff}, $E$ occurs with probability at least $1-\exp(-\eps N/2560)$;
importantly, this event only depends on $Y_1,\ldots,Y_N$, the first half of the uncontaminated sample.

In particular, on the event $E$,  for every $v \in S^{d-1}$,
\[
 \beta_v -\inr{\mu,v} + Q \le Q_0 + Q \le 5Q_0
\]
and
\[
 \beta_v -\inr{\mu,v} + Q \ge  \alpha_v -\inr{\mu,v} + Q \ge - Q_0 + Q \ge Q_0~.
\]
By a similar argument one may obtain lower and upper bounds for $\alpha_v -\inr{\mu,v}$. Hence, on $E$, for every $v \in S^{d-1}$,
\begin{equation} \label{eq:up-low-alpha-beta}
-5Q_0 \leq \alpha_v-\inr{\mu,v}-Q \leq -Q_0, \ \ {\rm and} \ \ Q_0 \leq \beta_v-\inr{\mu,v}+Q \leq 5Q_0~.
\end{equation}

Finally, recall that
$$
U_Q(v)=\frac{1}{N} \sum_{i=1}^N \phi_{\alpha_v-Q,\beta_v+Q}\left(\inr{\wt{X}_i,v}\right)~,
$$
and in order to complete the proof of Proposition \ref{thm:main-vector}, it suffices to show
that $U_Q(v)$ is uniformly close to $\inr{\mu,v}$, with high probability.
In particular, the next lemma implies Proposition \ref{thm:main-vector}.

\begin{lemma} \label{thm:uniform-approx}
Let $2Q_0 \leq Q \leq 4Q_0$. Conditioned on the event $E$,
with probability at least $1-2\exp(-\eps N/2560)$,
$$
\sup_{v \in S^{d-1}} \left| U_{Q}(v) - \inr{\mu,v} \right| \leq 2 \eps Q~.
$$
\end{lemma}

\begin{proof}
We prove that
$$
\sup_{v \in S^{d-1}} \left( U_{Q}(v) - \inr{\mu,v} \right) \leq 2\eps Q
$$
holds with the wanted probability; the proof that
$$
\sup_{v \in S^{d-1}} \left( \inr{\mu,v} - U_{Q}(v) \right) \leq 2\eps Q
$$
follows an identical argument and is omitted.

As a first step, note that, in the expression of $U_Q(v)$, the
corrupted samples $\wt{X}_i$ may be harmlessly replaced by their
uncorrupted counterparts $X_i$. Indeed, by
 \eqref{eq:up-low-alpha-beta}, on the event $E$, the range of the function
$\phi_{\alpha_v-Q,\beta_v+Q}$
is an interval of length at most $10Q$ and therefore,
deterministically, for all $v\in
S^{d-1}$,
\[
  \frac{1}{N} \sum_{i=1}^N
  \phi_{\alpha_v-Q,\beta_v+Q}\left(\inr{\wt{X}_i,v}\right)
- \frac{1}{N} \sum_{i=1}^N
\phi_{\alpha_v-Q,\beta_v+Q}\left(\inr{X_i,v}\right)
\le \eta \cdot 10Q \le \eps Q~.
\]
Once again, recalling that on $E$ \eqref{eq:up-low-alpha-beta}
holds, it follows that
\[
\frac{1}{N} \sum_{i=1}^N
\phi_{\alpha_v-Q,\beta_v+Q}\left(\inr{X_i,v}\right)
\le \frac{1}{N} \sum_{i=1}^N
\phi_{\inr{\mu,v}-Q_0, \inr{\mu,v}+5Q_0}\left(\inr{X_i,v}\right)~.
\]
Since the event $E$ only depends on the uncorrupted sample
$Y_1,\ldots,Y_N$, the right-hand side of the above inequality is
independent of $E$.
Thus, writing
\[
\ol{U}_Q(v) = \frac{1}{N} \sum_{i=1}^N \phi_{\inr{\mu,v}-Q_0,
  \inr{\mu,v}+5Q_0}(\inr{X_i,v})-\inr{\mu,v}
= \frac{1}{N} \sum_{i=1}^N \phi_{-Q_0, 5Q_0}(\inr{X_i-\mu,v})~,
\]
it suffices to prove that, with probability at least
$1-2e^{-\eps N/2560}$,
$$
\sup_{v \in S^{d-1}} \ol{U}_Q(v) \leq \eps Q~.
$$
To that end, consider the decomposition
\[
\sup_{v \in S^{d-1}} \ol{U}_Q(v)
\leq  \sup_{v \in S^{d-1}} \left( \ol{U}_Q(v)- \E \ol{U}_Q(v) \right)+   \sup_{v \in S^{d-1}} \E \ol{U}_Q(v) \defeq (1) + (2)~.
\]
First, let us bound the term $(1)$ in several steps.

Set
$$
\ol{W}_Q(v) = \frac{1}{N}\sum_{i=1}^N \phi_{-3Q,3Q}(\inr{X_i-\mu,v})~,
$$
and note that
\begin{eqnarray*}
\sup_{v \in S^{d-1}} \left( \ol{U}_Q(v)- \E \ol{U}_Q(v) \right)
& \le & \sup_{v \in S^{d-1}} \left( \ol{U}_Q(v)- \ol{W}_Q(v) \right) \\
& & + \sup_{v \in S^{d-1}} \left( \ol{W}_Q(v)- \E \ol{W}_Q(v) \right) \\
& & + \sup_{v \in S^{d-1}} \left( \E \ol{W}_Q(v)- \E \ol{U}_Q(v)
    \right) \\
& \defeq & (a) + (b) + (c)~.
\end{eqnarray*}
To bound term $(a)$, recall that $2Q_0 \le Q \le  4Q_0$, implying that
$\phi_{-Q_0,5Q_0}(x) \not = \phi_{-3Q,3Q}(x)$ only if
\[
\text{either} \quad
 x < -Q_0, \quad \text{or} \quad x > 5Q_0~.
\]
In both cases
$$
\left|\phi_{-Q_0,5Q_0}(x) - \phi_{-3Q,3Q}(x)\right| \leq 3Q~.
$$
By Lemma \ref{lemma:uniform-cutoff}, with probability at least $1-\exp(-\eps N/2560)$,
$$
\sup_{v \in S^{d-1}} \left|\{i : \inr{X_i-\mu,v} > 5Q_0 \ \ {\rm or} \ \  \inr{X_i-\mu,v} < -Q_0\}\right| \leq \frac{\eps N}{4}~,
$$
hence, on this event,
$$
(a) \leq \frac{3\eps Q}{4}~.
$$
One may control term $(c)$ similarly. For each $v\in S^{d-1}$,
\[
\E \ol{W}_Q(v)- \E \ol{U}_Q(v) \le 3Q \cdot \PROB\{|\inr{X-\mu,v}|>
Q_0\} \le \frac{3\eps Q}{64}
\]
by recalling \eqref{eq:chebyshev}.

The term $(b)$ is controlled using Talagrand's concentration inequality for the
supremum of empirical processes. Note that for every $v \in S^{d-1}$,
$$
\left|\phi_{-3Q,3Q}\left(\inr{\ol{X},v}\right)\right| \leq 3Q \quad \text{and}
\quad \EXP \left| \phi_{-3Q,3Q}\left(\inr{\ol{X},v}\right)\right|^2 \leq  \EXP
\left|\inr{\ol{X},v}\right|^2  \leq \lambda_1~.
$$
Also, since $\phi_{-3Q,3Q}(x)$ is a $1$-Lipschitz function that passes
through $0$, by a contraction argument (see Ledoux and Talagrand \cite{LeTa91}),
$$
\E \sup_{v \in S^{d-1}} \left|\ol{W}_Q(v) - \E \ol{W}_Q(v)\right| \leq 2\E \sup_{v \in S^{d-1}} \left|\frac{1}{N}\sum_{i=1}^N \eps_i \inr{X_i-\mu,v} \right| \leq 2 \sqrt{\frac{\Tr(\Sigma)}{N}}~.
$$
Hence, by Talagrand's inequality, with probability at least $1-2\exp(-x)$,
$$
\sup_{v \in S^{d-1}} \left|\ol{W}_Q(v) - \E \ol{W}_Q(v)\right|
\leq 4 \sqrt{\frac{\Tr(\Sigma)}{N}}+ 2\sqrt{\lambda_1} \sqrt{\frac{x}{N}} + 20Q\frac{x}{N} \leq \frac{\eps Q}{64}
$$
with the choice of $x = \eps N/2560$, recalling the definition of
$Q_0$, and using that $Q \ge 2Q_0$. This concludes the proof that $(1) \leq (1/2+1/32+1/400)\eps Q$ with probability $1-e^{-\eps N/2560}$.

Finally, it remains to estimate term $(2)$:
$$
(2)=\sup_{v \in S^{d-1}} \E \ol{U}_Q(v) = \sup_{v \in S^{d-1}}\E \phi_{-Q_0, 5Q_0} \left(\inr{\ol{X},v}\right)~.
$$
Clearly $X_v=\inr{\ol{X},v}$ is centered and
$\phi_{-Q_0, 5Q_0} (X_v) \not = X_v$ only when either $X_v \geq 5Q_0$ or $X_v \leq -Q_0$. Hence,
\begin{eqnarray*}
 \E \phi_{-Q_0, 5Q_0} (X_v)
& = &
 \E \left(\phi_{-Q_0, 5Q_0} (X_v) - X_v\right) \\
& \le & \EXP |Q_0+X_v| \IND_{X_v \le -Q_0}
\\
& \leq & \frac{\eps Q}{64}
\end{eqnarray*}
by an argument analogous to \eqref{eq:ebound} and using \eqref{eq:chebyshev}.
\end{proof}

With Proposition \ref{thm:main-vector} proved, let us complete
the proof of Theorem \ref{thm:main}.
Let $i_0$ be such that $Q \defeq 2^{i_0} \in [2Q_0,4Q_0)$ and let $E$ be the ``good'' event
that both \eqref{eq:uniform-beta} and
$$
\sup_{v \in S^{d-1}} \left| U_{Q}(v) - \inr{\mu,v} \right| \leq 2 \eps Q
$$
hold. Recall that
$$
U_Q(v)=\frac{1}{N} \sum_{i=1}^N \phi_{\alpha_v-Q,\beta_v+Q}\left(\inr{\wt{X}_i,v}\right)~;
$$
$E$ holds with probability at least $1-\delta$; and on $E$,
any point in $\Gamma(2^{i_0})$ is within distance $4\eps Q_0$ of the mean
$\mu$.
Hence, it suffices to show that on the event $E$, the sets $\Gamma(2^i)$ for $i \geq
i_0$ are nested. Indeed, by the definition of $i^*$,
$$
\emptyset \neq \bigcap_{i \geq i^*} \Gamma(2^i) \subset \Gamma(2^{i_0}),
$$
and thus $\|\wh{\mu}-\mu\| \leq 4\eps Q_0$.

To see that $\Gamma(2^{i_0})\subset \Gamma(2^{i_0+1})$  it is enough
to show that, for all $v \in S^{d-1}$,  $|\inr{x,v}-U_{2Q}(v)| \leq 4\eps Q$.
But if $x \in \Gamma(v,Q)$ for some $v\in S^{d-1}$, it follows that
$$
|\inr{x,v}-U_{2Q}(v)|  \leq  |\inr{x,v}-U_Q(v)|+|U_Q(v)-U_{2Q}(v)| \leq 2\eps Q + |U_Q(v)-U_{2Q}(v)|;
$$
therefore, it suffices to show that $|U_Q(v)-U_{2Q}(v)| \leq 2 \eps Q$.

Note that on the event $E$, there are at most $\eps N/4$ sample points
$\wt{X}_i$ such that $\inr{\wt{X}_i,v}$ is above or below the levels $\alpha_v-2^{i_0}$ and
$\beta_v+2^{i_0}$.
Hence, the number of points for which $U_Q(v) \not = U_{2Q}(v)$ is at
most $\eps N/4$ and so the difference is at most $(2Q \eps N/4)/N
= \eps Q/2$.

By induction, the same argument shows that, on the event $E$,
$\Gamma(2^i) \subset \Gamma(2^{i+1})$ for every $i \geq i_0$,
completing the proof of Theorem \ref{thm:main}.

\subsection*{Acknowledgements}
We thank the referees and the associate editor for  insightful
comments and pointing out relevant connections to
previous work.

\end{document}